\documentclass[11pt,a4paper]{article}
\usepackage{setspace}
\usepackage{fullpage}

\usepackage[latin1]{inputenc}
\usepackage{amsthm, amssymb, amscd, amsmath}
\usepackage{hyperref}

\newtheorem{theorem}{Theorem}
\newtheorem{lemma}{Lemma}
\newtheorem{proposition}{Proposition}
\newtheorem{corollary}[theorem]{Corollary}
\newtheorem{remark}{Remark}

\newcommand{\sinc}{\text{sc}}
\newcommand{\var}{\mbox{\rm Var}}
\newcommand{\cov}{\text{cov}}
\newcommand{\Prob}{\mathbb P}
\newcommand{\Esp}{\mathbb E}
\newcommand{\eps}{\varepsilon}

\newcommand{\R}{\mathbb R}
\newcommand{\Hb}{\boldsymbol H}
\newcommand{\indicator}{\mathbf{1}}

\author{Jean-Marc Aza\"{i}s\thanks{Mailing address: Universit\'{e} de Toulouse, 
IMT, ESP, F31062 Toulouse Cedex 9, France. Email: jean-marc.azais@math.univ-toulouse.fr
}
\qquad
Federico Dalmao\thanks{
Departamento de Matem\'{a}tica y Estad\'{i}stica del Litoral, 
Universidad de la Rep\'{u}blica, A.P. 50000, Salto, Uruguay. 
E-mail: fdalmao@unorte.edu.uy.}
\qquad
Jos\'{e} R. Le\'{o}n\thanks{Escuela de Matem\'{a}tica. Facultad de Ciencias. 
Universidad Central de Venezuela. 
A.P. 47197, Los Chaguaramos, Caracas 1041-A, Venezuela. Email: jose.leon@ciens.ucv.ve
}}
\title{CLT for the zeros of Classical Random Trigonometric Polynomials}
\begin{document}
\maketitle
\begin{abstract}
We prove a Central Limit Theorem for the number of zeros of random
trigonometric polynomials of the form $K^{-1/2}\sum_{n=1}^{K} a_n\cos(nt)$, 
being $(a_n)_n$ independent standard Gaussian random variables. 
In particular we show that the variance is equivalent to $V^2K\pi$, $0<V^2<\infty$, as $K\to\infty$. 
This last result was recently proved by Su \& Shao in \cite{SuShao}.
Our approach is based on the Hermite/Wiener-Chaos decomposition for square-integrable functionals 
of a Gaussian process and on Rice Formula for zero counting.\\

Nous montrons un Th\'{e}or\`{e}me de la Limite Central pour le nombre de racines d'un polyn\^{o}me trigonom\'{e}trique al\'{e}atoire 
de la forme  {$K^{-1/2}\sum_{n=1}^Ka_n\cos(nt)$}, ici les $a_n$ sont des variables al\'{e}atoires Gaussiennes standard et ind\'{e}pendantes. 
En particulier, nos d\'{e}montrons que  la variance asymptotique du nombre de racines est \'{e}quivalent \`{a} $V^2K\pi$, pour une certaine constante $V>0$, lorsque $K\to\infty$. 
Ce dernier r\'{e}sultat a \'{e}t\'{e} r\'{e}cemment d\'{e}montr\'{e} par Su \& Shao dans \cite{SuShao}.
Notre approche utilise la d\'{e}composition dans le chaos d'It\^{o}-Wiener d'une fonctionnelle non lin\'{e}aire de carr\'{e} int\'{e}grable et la formule de Rice.
\end{abstract}

\section{Introduction and Main Result}
Consider the classical random trigonometric polynomials, 
that is, the polynomials defined, for $K=1,2,\dots$, by
\begin{equation}\label{f:tk}
T_{K}(t):=\frac{1}{\sqrt{K}}\sum^{K}_{n=1}a_{n}\cos(nt),
\end{equation}
where the coefficients $a_n$ are i.i.d. standard Gaussian random variables 
and $t\in[0,\pi]$. 
%

%
One of the main questions about random polynomials concerns 
the random variable: {\it number of zeros} (level crossings in general); 
in our case the number of zeros of $T_K$ 
on the interval $[0,\pi]$. 
The distribution of this random variable remains unknown, 
and one way to approach it is to describe it 
not directly but through its moments. 

The asymptotic expectation of the number of zeros of classical trigonometric polynomials 
on $[0,\pi]$ is known since Dunnage's work \cite{dunnage} to be $K/\sqrt{3}$. 
The variance of $T_K$ was conjectured by Farahmand \& Sambandham \cite{farahmandcov} and Granville \& Wigman \cite{granville} 
to be equivalent to $V^2 K\pi$, as $K$ grows to infinity, 
for some constant $V^2>0$.
Recently  Su \& Shao  in \cite{SuShao} give a positive answer to this conjecture. 
They obtain the exact asymptotic behavior of the variance of the number of roots on $[0,\pi]$ for the classical trigonometric polynomials.

In this article we revisit the above result. We exhibit a simplified proof. 
Furthermore, and this is the main result of our work, we establish a Central Limit Theorem (CLT) for the number of zeros of $T_K$ 
on $[0,\pi]$ as $K\to\infty$. 

Observe that since $\cos(t)=\cos(2\pi-t)$ the number of zeros of $T_K$ on $[0,2\pi]$ is exactly twice the number of zeros of $T_K$ on $[0,\pi]$. 
Hence, all the results can be trivially adapted to $[0,2\pi]$.
Furthermore, 
it is easy to see that $T_K$ and its derivatives are jointly non-degenerated Gaussian. 
Hence $T_K$ has not multiple zeros almost suerely.

Our main result is the following. 
Let  us denote the number of zeros of  a Gaussian process $Z : \R\to \R$ on the interval $I$ by $N_{Z}(I):=\#\{t\in I:\, Z(t)=0\}$. 

\begin{theorem}\label{main}
The normalized number of zeros of $T_K$ on the interval $[0,\pi]$ converges in distribution 
to a Gaussian random variable. 
More precisely
\begin{equation*}
\frac{N_{T_K}[0,\pi]-\Esp\left(N_{T_K}[0,\pi]\right)}{\sqrt{\pi K}}\Rightarrow N(0,V^2), 
\end{equation*}
as $K\to\infty$, being $0<V^2<\infty$. 
\end{theorem}
In particular, we prove that 
\begin{equation*}
\lim_{K\to\infty}\frac{\var(N_{T_K}([0,K\pi]))}{K\pi}=\lim_{K\to\infty}\frac{\var(N_{X_K}([0,K\pi]))}{K\pi}, 
\end{equation*}
where $X_K$ are 
the stationary trigonometric polynomials defined by
\begin{equation}\label{Xk}
X_{K}(t)=\frac{1}{\sqrt{K}}\sum^{K}_{n=1}\left(a_{n}\cos(nt)+b_n\sin(nt)\right),
\end{equation}
being $a_n$, $b_n$ i.i.d. standard Gaussian random variables. 
The stationary trigonometric polynomials were studied by Granville \& Wigman \cite{granville}, Aza\"{i}s \& Le\'{o}n \cite{cltazaisleon}, 
see also the references therein. 
By \cite[Theorem 4.2]{cltazaisleon} and \cite[Eq (4)]{granville} we deduce that $\var(N_X([0,\pi]))$ is equivalent to $V^2K\pi$ with 
\begin{equation*}
V^2\approx 0.089.
\end{equation*}
As said above, Su \& Shao \cite{SuShao} show that the variance 
of the number of zeros of $T_K$ is equivalent to $cK\pi$ as $K\to\infty$ 
where $c$ is a complicated constant, $c\approx 0.257$. 

Therefore, our result is not in concordance with this value. 
We checked our result by numerical simulations.\\

As a by product of our proof, we obtain the following CLT 
for the number of zeros of a non-stationary Gaussian process 
on $[0,\infty)$ as the interval increases. 
The cardinal sine function is defined by 
\begin{equation*}
\sinc(x)=\frac{\sin(x)}{x}. 
\end{equation*}
\begin{corollary}\label{cor}
Let $T$ be the centered Gaussian process on $[0,\infty)$ 
with covariance function $r(s,t)=\textonehalf(\sinc(t-s)+\sinc(t+s))$.
Then, the number of zeros of $T$ on the interval $[0,K\pi]$ 
converges in distribution, after standardization, towards a standard Gaussian random variable.
\end{corollary}

\noindent{\bf Background:} 
Classical random trigonometric polynomials appears in Phy-sics, 
for instance in nuclear physics (random matrix theory), statistical mechanics, 
quantum mechanics, theory of noise, see Granville \& Wigman \cite{granville} and references therein. 

\noindent {\bf Previous works:}
The number of zeros of different ensembles of random polynomials have attracted attention 
of physicists and mathematicians
for at least seventy years. 
Consequently, there is an extensive literature on the subject, 
starting with Littlewood \& Offord \cite{lo1, lo2, lo3} who studied algebraic polynomials, 
these works were complemented by those of Erd\"{o}s \& Offord \cite{erdos-offord}, Kac \cite{kac-roots} and 
Ibragimov \& Maslova \cite{ibragimov-maslova1, ibragimov-maslova2}.
The final result for the mean number of roots of an algebraic polynomial of degree $K$, 
is that, for $i.i.d.$ coefficients in the domain of attraction of the normal law,
the mean number of roots is equivalent to $2\log(K)/\pi$ 
for centered $a_n$ and half this quantity for non-centered $a_n$. 
Maslova \cite{maslova1, maslova2} established the asymptotic variance and a CLT 
for the number of zeros of algebraic polynomials.

%
In the case of $X_K$ the stationarity and the Gaussianity simplify largely the treatment of the level crossing counting problem. 
For this ensemble of (stationary) trigonometric polynomials 
Granville \& Wigman \cite{granville} gave a proof of the CLT for the number of zeros, 
using conditions on moments of order higher than the second. 
After that, Aza\"{i}s \& Le\'{on} \cite{cltazaisleon} extended this result to all levels and gave a simplified proof of the CLT 
lying on the Wiener Chaos decomposition and Taqqu-Peccati's method. 
In particular, avoiding conditions on higher moments than the second. 

%
Finally, for Classical Trigonometric Polynomials, 
Wilkins \cite{wilkins} proved that 
\begin{equation*}
\Esp\left(N_{T_K}[0,2\pi]\right)=\frac{1}{\sqrt{3}}\left[(2K+1)+D_1+\frac{D_2}{2K+1}+\frac{D_3}{(2K+1)^2}\right]+O\left(\frac{D_3}{(2K+1)^3}\right)
\end{equation*}
with $D_1=0.23\dots$, etc.

The leading asymptotic term was proven to be the same for non-centered coefficients and for dependent coefficients 
with constant correlation, see Farahmand \cite{farahmandvar} and references therein. 

Recently, Farahmand \& Li \cite{farahmandli} considered the mean number of roots of $T_K$ and $X_K$ 
allowing its coefficients to have different means and variances, but being independent and Gaussian. \\


\noindent{\bf The techniques:} 
In order to obtain our results we make use of the techniques of Wiener Chaos expansion, 
Peccati-Tudor's method for obtaining CLTs and Rice Formula. 

More precisely, 
we obtain the Wiener Chaos expansion for the normalized number of zeros of 
the normalized version of $T_K$ 
on the interval $[0,\pi]$. 
Then, we use this expansion in order to obtain the order of the asynptotic variance.

A key fact, is that 
removing the extremes of the interval, 
the behavior of (the covariance of the standardized version of) $T_K$ is very similar to that 
of (the covariance of) its stationary counterpart $X_K$, 
so the limit variances of the respective number of zeros coincide. 
Our work closely follows Aza\"{\i}s \& Le\'{o}n \cite{cltazaisleon}, 
but the asymptotic Gaussianity is obtained directly from Peccati-Tudor's method  
rather than through the $L^2$ proximity to the limit process.

In that sense, it is worth to remark that whereas in the case of stationary trigonometric polynomials 
the CLT is inherited from that of the limit process, 
in our case the CLT for the limit process is a consequence of the CLT for classical trigonometric polynomials.

The article is organized as follows. It contains three sections. The present one introducing the subject  and establishing  the main result. The second one contains the proof of the main result. This section is split into two subsections. The first deals with the expansion into the Wiener Chaos of the crossings and the second contains the proof of the central limit theorem for these crossings once centered and  normalized. The third section contains the proof of five useful lemmas and also includes a notation table. 

\section{Proof}
\subsection{Expansion into the Chaos.}
Let us replace $t$ by $t/K$, this permits us to look at the polynomials $T_K$ 
at a convenient scale and to find a limit for them. 
Thus, from now on we are concerned with the zeros on the interval $[0,K\pi]$ of
\begin{equation} \label{f:tt}
\widetilde T_K(t)=\frac{1}{\sqrt{K}}\sum^K_{n=1}a_n\cos\left(\frac{n}{K}t\right).   
\end{equation}
Similarly, let $ \widetilde X_K(t)=X_K(t/K)$, with $X_K$ defined in Equation \eqref{Xk}. 
Clearly $ \widetilde T_K$ and $ \widetilde X_K$ are centered Gaussian processes in $t\in[0,K\pi]$.

Denote by $c_K$ the covariance function of $ \widetilde X_K$, 
it is well known, see Aza\"{i}s \& Le\'{o}n \cite[Eq. 1]{cltazaisleon}, Granville \& Wigman \cite[Eq. 13]{granville}, that
\begin{equation}\label{cK}
c_K(t-s)=\Esp\left( \widetilde X_K(s) \widetilde X_K(t)\right)=\frac{1}{K}\sum^{K}_{n=1}\cos\left(\frac{n}{K}(t-s)\right).
\end{equation}
Note that $c_K(\cdot)$ can be seen as a Riemann sum for $\int^1_0\cos(u\cdot)du$. 
Further, $c_K$ can be expressed in closed form using Fej\'{e}r Kernel. 

A direct computation shows that 
the covariance function of the classical trigonometric polynomials, 
$r_K(s,t):=\Esp\left( \widetilde T_K(s) \widetilde T_K(t)\right)$, is given by
\begin{equation}\label{rK}
r_K(s,t)=\frac{1}{2}(c_K(t-s)+c_K(t+s)).
\end{equation}
In particular, the variance of $ \widetilde T_K(t)$ is 
\begin{equation*}
V^2_K(t):=r_K(t,t)=\frac{1}{2}(1+c_K(2t)). 
\end{equation*}
Thus, the limit variance, as $t\to\infty$, is $\textonehalf$ and not $1$ as in the stationary case. 
Needless to say that the processes  obtained from the $ \widetilde T_K$'s replacing the cosines by sines 
have also asymptotic variance $\textonehalf$. 
At $t=0$, $ \widetilde T_K$ and $ \widetilde X_K$ have the same variance.

From now on, we work on the case $s<t$ and denote $\tau:=t-s$ and $\sigma:=t+s$. 

It is convenient to use the standardized version of $ \widetilde T_K$, namely 
\begin{equation}\label{def}
\overline{T}_K(t):=\widetilde T_K(t)/V_K(t),
\end{equation} 
thus $\overline{T}_K$ has unit variances and its covariance, $\overline{r}_K$, is given by
\begin{equation}\label{barr}
\overline{r}_K(s,t)=\frac{c_K(\tau)+c_K(\sigma)}{\sqrt{1+c_K(2s)}\sqrt{1+c_K(2t)}}.
\end{equation}

\noindent{\bf Limit covariances:} 
We collect in the following lemma some basic properties of $c_K$ 
taken from Aza\"{i}s \& Le\'{o}n \cite[Eq (2.5)-(2.8)]{cltazaisleon}.
\begin{lemma}\label{lema:cK}
For $\tau\in[0,K\pi]$ we have
\begin{enumerate}
 \item as $K\to\infty$. 
\begin{equation}\label{limitck}
c_K(\tau)\to\sinc(\tau). 
\end{equation}
Furthermore, this convergence is uniform in compacts not containing $0$. 
Besides, the order one and order two derivatives of $r_K$ converge in the same manner to the 
corresponding derivatives of $r$.
\item for some constant $c$
\begin{multline}\label{cota:supr1}
|c_K(\tau)|\leq\frac{\pi}{\tau},\quad
|c^{\prime}_K(\tau)|\leq\frac{\pi}{\tau}+\frac{\pi^2}{2\tau^2},\\
|c^{\prime\prime}_K(\tau)|\leq c\left(\frac{1}{\tau}+\frac{1}{\tau^2}
+\frac{1}{\tau^3}\right).
\end{multline}
\end{enumerate}
\end{lemma}
We also need the following bounds.
\begin{lemma}\label{cotabaf}
The variances of $T$ and $ \widetilde T_K$ are bounded away from zero in $[0,K\pi]$, 
those of $\overline{T}^{\prime}$ and $\overline{T}^{\prime}_K$ are bounded away from zero 
in $[t_0,K\pi-t_0]$ for $t_0$ large enough. 
These bounds are uniform in $K$.
\end{lemma}

From Equations \eqref{cK}, \eqref{rK} and Lemma \ref{lema:cK} it follows that
\begin{equation}\label{limitsc}
r_K(s,t)\to r(s,t):=\frac{1}{2}(\sinc(\tau)+\sinc(\sigma)), 
\end{equation}
as $K\to\infty$ for all $s,t\in[0,K\pi]$ and that this convergence is 
uniform in off-diagonal compacts (compacts not containing points in the diagonal $s=t$). 
Furthermore, the order one and order two derivatives of $r_K$ converge in the same manner to 
the corresponding derivatives of $r$ respectively. 
Besides the following bounds hold 
for $s,t\in[0,K\pi]$ 
\begin{multline}\label{cota:supr}
|r_K(s,t)|\leq\frac{\pi}{\tau}+\frac{\pi}{\sigma},\quad
|\partial_{i}r_K(s,t)|\leq\frac{\pi}{2\tau}+\frac{\pi}{2\sigma}+\frac{\pi^2}{4\tau^2}+\frac{\pi^2}{4\sigma^2},\\
|\partial_{ij}r_K(s,t)|\leq c\left(\frac{1}{\tau}+\frac{1}{\sigma}+\frac{1}{\tau^2}+\frac{1}{\sigma^2}
+\frac{1}{\tau^3}+\frac{1}{\sigma^3}\right),
\end{multline}
where $i,j=s,t$ and $c$ is some constant; 
and
\begin{equation}\label{cota:supvar}
V^{2}_{K}(t)\leq\frac{1}{2}\left[1+\frac{\pi}{2t}\right].
\end{equation}

Finally, observe that for $t_0$ large enough and $s,t\in[t_0,K\pi-t_0]$,
$\overline{r}_K(s,t)$ and its derivatives converge to 
\begin{equation}\label{limitscbar}
\overline{r}(s,t)=\frac{\sinc(\tau)+\sinc(\sigma)}{\sqrt{1+\sinc(2s)}\sqrt{1+\sinc(2t)}},
\end{equation}
and its derivatives respectively uniformly in off-diagonal compacts. 
This is enough for our purposes. 
Therefore, $ \widetilde T_K$, $\overline{T}_K$ converge (in the described sense) to centered Gaussian processes $T$ and $\overline{T}$ 
on the positive real axis having, respectively, covariances given by \eqref{limitsc} and \eqref{limitscbar}.
In the sequel we need also  to renormalized $\overline{T}'_k$ in order  to set its varaince to 1. So we define the standard deviation of $\overline{T}^{\prime}_K(s)$ :
\begin{equation}\label{vK}
v_K(s):=\sqrt{\overline{r}^{(1,1)}_{K}(s,s)},
\end{equation}
where $  r_K^{(i,j)}(s,t) :=\frac{ \partial ^{i+j}}{\partial^i s \partial^j t } r_K$ 
and 
\begin{equation}\label{f:cal}
\mathcal T^{\prime}_K(s)=\overline{T}^{\prime}_K(s)/v_{K}(s).
\end{equation}

\begin{remark}
The role of the limit process in this work is secondary, 
compare with Aza\"{i}s \& Le\'{o}n \cite{cltazaisleon} and Granville \& Wigman \cite{granville}. 
\end{remark}
\noindent{\bf Chaining and Isonormal Process:} 
For the sake of readability, we write all the trigonometric polynomials on the same probability space. 
Note that, since we only care about distributions in the sequel, this is not really necessary.

Let $B=(B_\lambda)$ be a standard Brownian Motion defined on a probability space $(\Omega,{\cal F},\Prob)$. 
We assume that ${\cal F}$ is generated by $B$. 
By the isometric property of the stochastic integral, 
we have 
\begin{equation} \label{f:gam}
 \widetilde T_K(t)=\int^1_0\gamma_{K}(t,\lambda)dB_{\lambda},\quad
\gamma_K(t,\lambda)=\sum^{K}_{n=1}\cos\left(\frac{n}{K}t\right)\indicator_{\left[\frac{n-1}{K},\frac{n}{K}\right)}(\lambda) 
\end{equation}
being $\indicator_A$ the indicator (or characteristic) function of the set $A$. 
Thus, the processes $ \widetilde T_K:K=1,2,\dots$ are defined on the same probability space. 

Furthermore, 
let $\Hb$ be the Hilbert space $L^2([0,1],{\cal B},d\lambda)$, 
being ${\cal B}$ the Borel $\sigma$-algebra and $d\lambda$ the Lebesgue measure, with the usual inner 
product $\left\langle g,h\right\rangle=\int^1_0g(\lambda)h(\lambda)d\lambda$. 
Thus, the map
\begin{equation*}
h\mapsto B(h):=\int^1_0h(\lambda)dB_{\lambda},
\end{equation*}
defines an isometry between $\Hb$ and $L^2(\Omega)=L^2(\Omega,{\cal F},\Prob)$. 
In this situation, $(B(h):h\in\Hb)$ is called an isonormal process associated 
to $\Hb$.

In particular, we have
$$ \widetilde T_K(t)=B(\gamma_K(t,\cdot)),\,\overline{T}_K(t)=B(h_K(t,\cdot)), \mbox{ with } 
h_K(t,\cdot)=\gamma_K(t,\cdot)/V_K(t),$$ and $$\mathcal{T}^\prime_K(t):=B(h^\prime_K(t,\cdot)),\,
\mbox{ with } h^{\prime}_K(t,\cdot)=\partial_th_K(t,\cdot)/\|\partial_th_K(t,\cdot)\|_2.$$
In addition, $T(t)=B(\indicator_{[0,1]}(\cdot)\cos(t\cdot))$.\\

\noindent{\bf Chopping the extremes of the interval:}
The main idea of the proof of the CLT is to take advantage of the fact that the covariances of 
$\widetilde{T}_K$ and $ \widetilde X_K$ 
are very similar one to each other for large values of $s,t$, 
even $ \widetilde T_K$ being non-stationary. 
This idea is supported by the following lemma.

Denote $[0,K\pi]_{-{\alpha}}=[(K\pi)^{\alpha},K\pi-(K\pi)^{\alpha}]$ for $\alpha\in(0,1/2)$.
\begin{lemma}\label{lemma:truncar}
Let  $\alpha$  be any value, $0<\alpha<1/2$, (for example $\alpha =1/4$) , we have
\begin{equation*}
\frac{N_{\overline{T}_K}\left([0,K\pi]^{c}_{-{\alpha}}\right)-
\Esp\left(N_{\overline{T}_K}\left([0,K\pi]^{c}_{-{\alpha}}\right)\right)}{\sqrt{K\pi}}\longrightarrow 0 
\end{equation*}
in probability, as $K\to\infty$. 
\end{lemma}
\
Then, it suffices to prove that 
\begin{equation*}
\frac{N_{\overline{T}_K}\left([0,K\pi]_{-\alpha}\right)-\Esp\left(N_{\overline{T}_K}\left([0,K\pi]_{-\alpha}\right)\right)}
{\sqrt{K\pi}}\Rightarrow N(0,V^2).
\end{equation*}

\noindent{\bf Wiener Chaos decomposition:} 
Following Kratz \& Le\'{o}n \cite{kratz-leon-97} 
we can establish the expansion of the number of roots in the Wiener Chaos.\\

We need some facts about Hermite Polynomials, Multiple Wiener-It\^{o} Integrals 
and the Chaotic Expansions, 
see Peccati \& Taqqu \cite{taqqu-peccati} and Hiu-Hsiung Kuo \cite{hiu} for details on 
the definitions and the results listed below. 

The Hermite polynomials $H_q$ are defined by 
\begin{equation*}
H_{q}(x)=(-1)^{q}e^{x^2/2}\left.\frac{d^{q}}{dt^{q}}e^{-t^2/2}\right|_{t=x}.
\end{equation*}
The family $(H_q:q\geq 0)$ is a complete orthogonal system in $L^2(\R,\varphi(dx))$ 
with $\|H_q\|^2_2=q!$. Here $\varphi$ stands for the standard Gaussian density function. 
Besides, for $q\geq 1$, the $q$-fold 
multiple Wiener-It\^{o} integral w.r.t $B$, $I^{B}_q$, can be introduced  
as the linear isometry, 
between the symmetric tensor product $L^2_s([0,1]^{q})=\Hb^{\odot q}$ 
with the norm $\sqrt{q!}\|\cdot\|_{L^2([0,1]^q)}$ and its image in $L^2(\Omega)$, 
induced by
\begin{equation*}\label{eq:hqI}
I^{B}_q(h^{\otimes q})=H_q(B(h)).
\end{equation*}
for $h\in\Hb$, with unit norm and $h^{\otimes q}(\lambda_1,\dots,\lambda_q)=\prod^{q}_{k=1}h(\lambda_k)$. 

For $q\geq 1$, the $q$-th Wiener Chaos ${\cal H}_q$ is defined as the image of $I^B_q$. 
Furthermore, denoting the set of constants by ${\cal H}_0$, 
we have
\begin{equation*}
L^2(\Omega,{\cal F},\Prob)=\bigoplus^{\infty}_{q=0}{\cal H}_q,
\end{equation*}
where $\oplus$ indicates an orthogonal sum. 
In other words, 
for any square integrable functional, $F\in L^2(\Omega,{\cal F},\Prob)$, 
of the Brownian motion $B$, 
there exists a unique sequence of symmetric functionals $(f_q:q\geq 1)$ 
with $f_q\in L^{2}([0,1]^{q})$, such that
\begin{equation*}\label{eq:hermexp}
F-\Esp(F)=\sum^{\infty}_{q=1}I^{B}_{q}(f_q),
\end{equation*}
where the equality holds in the $L^2$ sense.\\

We now write the Wiener Chaos expansion for the number of crossings.
We need some notations. 
Denote 
\begin{equation} \label{f:akbk}
a_{2\ell}=\sqrt{\frac{2}{\pi}}\frac{(-1)^{\ell+1}}{2^{\ell}\ell!(2\ell-1)},
\quad
b_{k}=
\begin{cases}
(-1)^{k/2}(k-1)!!,&\textrm{ if }k\textrm{ is even}\\
0,&\textrm{ if }k\textrm{ is odd} 
\end{cases}
\end{equation}
where $(2n-1)!!=\prod^{n}_{j=1}(2j-1)$;
\begin{equation} \label{f:fq}
f_q(x,y)=\sum^{\left\lfloor q/2\right\rfloor}_{\ell=0}
b_{q-2\ell}a_{2\ell}H_{q-2\ell}(x)H_{2\ell}(y),
\end{equation}
\begin{equation} \label{f:iqtk}
I^{\overline{T}_K}_{q}\left([0,K\pi]_{-{\alpha}}\right)=
\frac{1}{\sqrt{K\pi}}
\int_{[0,K\pi]_{-{\alpha}}}f_{q}\left(\overline{T}_K(s),\mathcal {T}^{\prime}_K(s)\right)v_K(s)ds
\end{equation}

\begin{theorem}\label{theorem:expansion}
The following expansion holds in $L^2$
\begin{equation*}
\frac{N_{\overline{T}_K}\left([0,K\pi]_{-{\alpha}}\right)-\Esp\left(N_{\overline{T}_K}\left([0,K\pi]_{-{\alpha}}\right)\right)}{\sqrt{K\pi}}=
\sum^{\infty}_{q=1}I^{\overline{T}_K}_{q}\left([0,K\pi]_{-{\alpha}}\right).
\end{equation*}
In particular, since $b_1=0$, $I^{\overline{T}_K}_1([0,K\pi]_{-\alpha})=0$ for all $K$.
\end{theorem}
The following two lemmas are the tools for proving the above theorem. The first one shows that the second moment of the crossings of zero of $\overline{T}_K$ is finite.
\begin{lemma}\label{cota:vardiag}
There exists $t_0>0$ such that, 
for all $a>0$, the variances of $N_{\overline{T}_K}([t,t+a])/\sqrt{K\pi}$ remain 
bounded uniformly on $K$ for all $t\in[t_0,K\pi-t_0]$.
\end{lemma}

\begin{lemma}\label{lemma:chaos}
1. $\Esp\left((N_{\overline{T}_K}([0,K\pi]_{-{\alpha}}))^2\right)<\infty$.\\
2. Denote by $N_{\overline{T}_K}(u,I)$ the number of crossings through the level $u$ by $\overline{T}_K$ in the interval $I$, then 
$\Esp\left(N^2_{\overline{T}_K}(u,I)\right)$ is continuous in $u$.\\
3. Define 
\begin{equation}\label{Nsigma}
N^{\eta}_{\overline{T}_K}:=\int_{[0,K\pi]_{-\alpha}}\varphi_{\eta}(\overline{T}_K(s))
\left|\mathcal {T}^{\prime}_K(s)\right|v_K(s)ds,
\end{equation}
where $\varphi_\eta$ is the density of the  $N(0,\eta)$ distribution.
Then, $N^{\eta}_{\overline{T}_K}$ converges almost surely and in $L^2$ to $N_{\overline{T}_K}([0,K\pi]_{-{\alpha}})$ 
and $\Esp(N^{\eta}_{\overline{T}_K})^2\to\Esp(N_{\overline{T}_K}([0,K\pi]_{-{\alpha}}))^2$.\\
4. The random variable $N^{\eta}_{\overline{T}_K}$ has the chaotic expansion
\begin{equation*}
N^{\eta}_{\overline{T}_K}=\sum^{\infty}_{q=0}\sum^{\left\lfloor q/2\right\rfloor}_{\ell=0}
b^{\eta}_{q-2\ell}a_{2\ell}\int_{[0,K\pi]_{-\alpha}}H_{q-2\ell}\left(\overline{T}_K(s)\right)
H_{2\ell}\left(\mathcal {T}^{\prime}_K(s)\right)v_K(s)ds 
\end{equation*}
where $b^{\eta}_k$ are the Hermite coefficients of $\varphi_{\eta}$.
\end{lemma}
\begin{proof}We will prove now the Theorem \ref{theorem:expansion}.
For ease of notation we will drop for a while the interval dependence in the crossings random variable. 

Since the zeros are isolated, 
formally, we can write the  Kac Formula:
\begin{equation*}\label{kac:chaos}
N_{\widetilde{T}_K}=N_{\overline{T}_K} =  N_{\overline{T}_K} ([0,K\pi]_{-{\alpha}})  =\int_{[0,K\pi]_{-{\alpha}}}\delta_{0}\left(\overline{T}_K(s)\right)
\left|\mathcal {T}^{\prime}_K(s)\right|v_K(s)ds 
\end{equation*}
In order to give it a precise meaning, 
we approximate the delta function at $0$ by Gaussian kernels $\varphi_{\eta}$.
Consider the random variables $N^{\eta}_{\overline{T}_K}$.  
By Part 4 of Lemma \ref{lemma:chaos}
\begin{equation}\label{eq:nsigmachaos}
N^{\eta}_{\overline{T}_K}=\sum^{\infty}_{q=0}\sum^{\left\lfloor q/2\right\rfloor}_{\ell=0}
b^{\eta}_{q-2\ell}a_{2\ell}\int_{[0,K\pi]_{-\alpha}}H_{q-2\ell}\left(\overline{T}_K(s)\right)
H_{2\ell}\left(\mathcal {T}^{\prime}_K(s)\right)v_K(s)ds.
\end{equation}
Now, the idea is to pass to the limit this expansion as $\eta\to 0$ 
in order to obtain the expansion for $N_{\overline{T}_K}$.

First, observe that $b^{\eta}_{k}\to b_{k}$ (non-random), and that this is the only ingredient depending on $\eta$. 
Hence $N^{\eta}_{\overline{T}_K}\to N_{\overline{T}_K}$ almost surely. 
Now consider the sum of the first $Q$ terms in the r.h.s. of \eqref{eq:nsigmachaos}. 
By Fatou's Lemma we have
\begin{multline*}
\sum^{Q}_{q=0}\Esp\left(\left[\sum^{\left\lfloor q/2\right\rfloor}_{\ell=0}
b_{q-2\ell}a_{2\ell}\int_{[0,K\pi]_{-\alpha}}H_{q-2\ell}(\overline{T}_K(s))
H_{2\ell}(\mathcal {T}^{\prime}_K(s))v_K(s)ds\right]^2\right)\\
\leq\liminf_{\eta\to 0}\Esp\left(\left[N_{\overline{T}_K}^{\eta}\right]^2\right)=\Esp\left(N_{\overline{T}_K}^2\right).
\end{multline*}
In the last equality we used Part 3 of Lemma \ref{lemma:chaos}. 
Therefore, the right hand side of (\ref{eq:nsigmachaos}) has a limit, say ${\cal N}$, in $L^2$ (with $b_{q-2\ell}$ instead of $b^{\eta}_{q-2\ell}$). 
It remains to show that this limit is, effectively, $N_{\overline{T}_K}$.
The result follows writing
\begin{equation*}
\|N_{\overline{T}_K}-{\cal N}\|^{2}_{L^2}\leq 2\left[
\|N_{\overline{T}_K}-N^{\eta}_{\overline{T}_K}\|^{2}_{L^2}+\|N^{\eta}_{\overline{T}_K}-{\cal N}\|^{2}_{L^2}\right].
\end{equation*}
The first term in the right hand side tends to zero by Part 3 of Lemma \ref{lemma:chaos}. 
To show that the second term tends to zero, consider its chaotic expansion
\begin{equation*}
N^{\eta}_{\overline{T}_K}-{\cal N}=\sum^{\infty}_{q=0}\sum^{\left\lfloor q/2\right\rfloor}_{\ell=0}
\left(b^{\eta}_{q-2\ell}-b_{q-2\ell}\right)
a_{2\ell}\,J_q,
\end{equation*}
where we denote 
\begin{equation}\label{f:jq} 
J_q=\int_{[0,K\pi]_{-\alpha}}H_{q-2\ell}(\overline{T}_K(s))
H_{2\ell}(\mathcal {T}^{\prime}_K(s))v_K(s)ds.
\end{equation}
Note that $J_q$ does not depend on $\eta$. 
Then, for each $Q$ we have
\begin{multline*}
\|N^{\eta}_{\overline{T}_K}-{\cal N}\|^{2}_{L^2}
\leq 3\left[\sum^{\infty}_{q=Q+1}\Esp\left(\left[\sum^{\left\lfloor q/2\right\rfloor}_{\ell=0}b_{q-2\ell}a_{2\ell}J_q\right]^2\right)\right.\\
\left.
+\sum^{Q}_{q=0}\Esp\left(\left[\sum^{\left\lfloor q/2\right\rfloor}_{\ell=0}(b^{\eta}_{q-2\ell}-b_{q-2\ell})a_{2\ell}J_q\right]^2\right)+\sum^{\infty}_{q=Q+1}\Esp\left(\left[\sum^{\left\lfloor q/2\right\rfloor}_{\ell=0}b^{\eta}_{q-2\ell}a_{2\ell}J_q\right]^2\right)
\right] .
\end{multline*}
Now, take limit as $\eta\to 0$, 
the first term does not depend on $\eta$; 
the second one tends to $0$ since it is a finite sum and $b^{\eta}_{q-2\ell}\to_{\eta}b_{q-2\ell}$; 
the third term is $\|P_Q(N^{\eta}_{\overline{T}_K})\|^{2}$, 
where $P_Q$ is the orthogonal projection of the $L^2$ random variable $N^{\eta}_{\overline{T}_K}$ 
on the subspace $\oplus^{\infty}_{Q+1}{\cal H}_q$, 
thus by Part 3 of Lemma \ref{lemma:chaos} it converges with $\eta\to 0$ to $\|P_Q(N_{\overline{T}_K)}\|^{2}$ and tends to $0$ when $Q\to\infty$.

This proves the theorem.
\end{proof}

\begin{remark}\label{remark:g-q}
As shown in the chaining paragraph, $\overline{T}_K(t)$ and $\mathcal {T}^{\prime}_K(t)$ 
can be written as $B(h_K(t,\cdot))$ and $B(h^\prime_K(t,\cdot))$ for $h(t,\cdot),h^\prime(t,\cdot)\in\Hb$. 
Since $\overline{T}_K(t)$ and $\mathcal {T}^{\prime}_K(t)$ are orthogonal,  
so are by the isometry $h(t,\cdot),$ and $h^\prime(t,\cdot)$, then, 
using the multiplication formula \cite[Eq. 6.4.17]{taqqu-peccati} we get.  
\begin{align*}
H_{q-2\ell}\left(\overline{T}_K(t)\right)H_{2\ell}\left(\mathcal {T}^{\prime}_K(t)\right)
&=H_{q-2\ell}(B(h_K(t,\cdot)))H_{2\ell}(B(h^\prime_K(t,\cdot)))\\
&=I^B_q\left(h^{\otimes q-2\ell}_K(t,\cdot)\otimes h^{\prime\,\otimes 2\ell}_K(t,\cdot)\right)
\end{align*}
Therefore, by the Stochastic Fubini's Theorem, see Peccati \& Taqqu \cite[Section 5.13]{taqqu-peccati}, 
we have
$I^{\overline{T}_K}_q([0,K\pi]_{-\alpha})=I^{B}_{q}(g_q)$ for $g_q\in L^{2}([0,1]^q)$ given by
\begin{equation}\label{f-q}
g_q({\boldsymbol\lambda},K)=\frac{1}{\sqrt{K\pi}}\int_{[0,K\pi]_{-\alpha}}\sum^{\left\lfloor q/2\right\rfloor}_{\ell=0}
b_{q-2\ell}a_{2\ell}(h^{\otimes q-2\ell}_K(s,\cdot)\otimes h^{\prime\,\otimes 2\ell}(s,\cdot))({\boldsymbol\lambda})v_K(s)ds,
\end{equation}
with ${\boldsymbol\lambda}=(\lambda_1,\dots,\lambda_q)$. 
That is, the variable $I^{\overline{T}_K}_q([0,K\pi]_{-\alpha})$ belongs to the $q$-th Wiener Chaos ${\mathcal H}_q$. 
In particular, they are orthogonal for different values of $q$.
\end{remark}

\noindent{\bf Asymptotic variance of $I^{\overline{T}_K}_q\left([0,K\pi]_{-\alpha}\right)$:} 
Following the arguments in Aza\"{i}s \& Le\'{o}n \cite{cltazaisleon} 
we obtain the asymptotic variance of the number of zeros of $\overline{T}_K$ 
on $[0,K\pi]_{-\alpha}$ as $K\to\infty$. 
Furthermore, 
in Proposition \ref{conv:var} 
we show that the limit variance of $N_{\overline{T}_K}([0,K\pi]_{-\alpha})$ 
coincides with that of 
the number of zeros, $N_{X}([0,K\pi]_{-\alpha})$, of the stationary Gaussian process $X$. 

The following lemma gives a uniform upper bound (on $K$) for these variances. 

%
\begin{lemma}\label{var:unifacot}
There exists $t_0>0$ such that 
the variances of the normalized number of roots of $\overline{T}_K$ on the interval $[t_0,K\pi-t_0]$ 
\begin{equation*}
\var\left(\frac{N_{\overline{T}_K}[t_0,K\pi-t_0]-\Esp\left(N_{\overline{T}_K}[t_0,K\pi-t_0]\right)}{\sqrt{K\pi}}\right) 
\end{equation*}
are uniformly bounded on $K$.
\end{lemma}

We are ready to compare the limit variances of $N_{\overline{T}_K}\left([0,K\pi]_{-\alpha}\right)/\sqrt{K\pi}$ 
and $N_{X}\left([0,K\pi]_{-\alpha}\right)/\sqrt{K\pi}$.
\begin{proposition}\label{conv:var}
For $q\geq 2$, the variances of $I_q^{\overline{T}_K}\left([0,K\pi]_{-\alpha}\right)$ 
and $I_q^{X}\left([0,K\pi]_{-\alpha}\right)$ 
have the same limit, denoted by $\sigma^2_q(0)$.
\end{proposition}
\begin{remark}
The value of $V^2=\sum^{\infty}_{q=2}\sigma_q^2(0)$ is shown to be finite in Aza\"{i}s \& Le\'{o}n \cite{cltazaisleon}. 
Furthermore, in Granville \& Wigman it is shown to be approximately $0.089$. 
Furthermore, by Aza\"{\i}s \& Le\'{o}n \cite{cltazaisleon} this also coincides with the asymptotic variance 
(after the same standardization) of $N_{X_K}$.
\end{remark}
\begin{proof}
By Remark \ref{remark:g-q} $I_q^{\overline{T}_K}\left([0,K\pi]_{-\alpha}\right)$ is a 
Wiener-It\^{o} integral, then it is a centered random variable, moreover recalling that $\tau=t-s$ and $\sigma=t+s$, we have
\begin{multline*}
\var\left(I^{\overline{T}_K}_{q}\left([0,K\pi]_{-\alpha}\right)\right)=\Esp\left(\left[I^{\overline{T}_K}_{q}\left([0,K\pi]_{-\alpha}\right)\right]^2\right)\\
=\frac{1}{K\pi}\int^{K\pi-(K\pi)^{\alpha}}_{(K\pi)^{\alpha}}\int^{K\pi-(K\pi)^{\alpha}}_{(K\pi)^{\alpha}}\Esp\big[
f_{q}(\overline{T}_K(s),\overline{T}^{\prime}_K(s))f_{q}(\overline{T}_K(t),\overline{T}^{\prime}_K(t))\big]\\
\cdot v_K(s)v_K(t)\;ds\;dt\\
=\frac{1}{2}\int^{K\pi-2(K\pi)^{\alpha}}_{-K\pi+2(K\pi)^{\alpha}}g(\tau) d\tau 
\end{multline*}
where
\begin{equation*}
g(\tau)=\frac{1}{K\pi}
\int^{2K\pi-2(K\pi)^{\alpha}-|\tau|}_{2(K\pi)^{\alpha}+|\tau|}E_K(\tau,\sigma)
v_K((\sigma-\tau)/2)v_K((\sigma+\tau)/2)d\sigma
\end{equation*}
and $E_K(\tau,\sigma)$ is the expectation written in terms of $\tau$ and $\sigma$, 
namely 
$E_K(\tau,\sigma)=\Esp\big[
f_{q}(\overline{T}_K((\sigma-\tau)/2),\overline{T}^{\prime}_K((\sigma-\tau)/2))
f_{q}(\overline{T}_K((\sigma+\tau)/2),\overline{T}^{\prime}_K((\sigma-\tau)/2))\big]$. 

Applying the mean value theorem, there exists 
$ \widetilde{\sigma}_K\in[2(K\pi)^{\alpha}-|\tau|,2K\pi-2(K\pi)^{\alpha}+|\tau|]$, 
such that
\begin{equation*}
g(\tau)=\frac{2K\pi-4(K\pi)^\alpha-4|\tau|}{K\pi}E_K(\tau, \widetilde{\sigma}_K)
v_K(( \widetilde{\sigma}_K-\tau)/2)v_K(( \widetilde{\sigma}_K+\tau)/2).
\end{equation*}
A direct computation shows that
\begin{equation}\label{eq:vK2}
v^2_K(u)=\frac{1}{1+c_K(2u)}\left[c^{\prime\prime}_K(2u)-c^{\prime\prime}_K(0)
-\frac{(c^{\prime}_K(2u))^2}{1+c_K(2u)}\right]. 
\end{equation}
Note that by the Inequalities in \eqref{cota:supr} $c_K(2u)$ and its derivatives tend to $0$ 
when $u\to\infty$. 
Besides, $c^{\prime\prime}_K(0)=-\frac{(K+1)(2K+1)}{6K^2}$. 
Therefore, since $\tilde{\sigma}_K\to_K\infty$ we deduce that 
$v_K(\widetilde{\sigma}_K)\to\sqrt{-\sinc^{\prime\prime}(0)}=1/\sqrt{3}$ as $K\to\infty$. 

By Mehler formula, see Aza\"{i}s \& Wschebor \cite[Lemma 10.7]{cltazaisleon}, 
$E_K(\tau,\sigma)$ is a polynomial on 
the covariances of 
$(\overline{T}_K(s),\overline{T}^{\prime}_K(s))$ and $(\overline{T}_K(t),\overline{T}^{\prime}_K(t))$; 
thus, $E_K(\tau, \widetilde{\sigma}_K)$ is a polynomial on 
\begin{equation*}
(c_K(\tau)\pm c_K( \widetilde{\sigma}))/\sqrt{(1+c_K( \widetilde{\sigma}-\tau))(1+c_K( \widetilde{\sigma}+\tau))},
\end{equation*}
and its derivatives. 
Let $K\to\infty$, then, $ \widetilde{\sigma}_K\to\infty$, $c_K( \widetilde{\sigma}_K),c_K( \widetilde{\sigma}_K\pm\tau)\to 0$, 
$c_K(\tau)\to\sinc(\tau)$ 
and $E_K(\tau, \widetilde{\sigma})$ tends to the same polynomial but with $c_K(\tau)$ replaced by $\sinc(\tau)$ and 
$c_K( \widetilde{\sigma}_K)$ replaced by $0$. 

Consider now the centered stationary Gaussian process $X$ 
with covariance function at $(s,t)$ given by $\sinc(\tau)$. 
For $\var(I^X_q)$, 
analogous computations to the preceeding ones 
show 
that 
\begin{equation*}
\lim_{K\to\infty}\var(I^X_q([0,K\pi]_{-\alpha}))
=\lim_{K\to\infty}\var(I^{\overline{T}_K}_q([0,K\pi]_{-\alpha})),
\end{equation*}
 provided the necessary domination given by 
Lemma \ref{var:unifacot}. The result follows.
\end{proof}

\subsection{\ CLT for the number of roots of Classical Trigonometric Polynomials:} 
This section is devoted to prove our main result: Theorem \ref{main}. 

Recall that for $q=1$, the random variables $I^{\overline{T}_K}_1\left([0,K\pi]_{-\alpha}\right)=0$ for all $K$ because $b_1=0$ .

\noindent{\bf Asymptotic Gaussianity of $I^{\overline{T}_K}_q$ for $q>1$:}  
We use the Fourth Moment Theorem by Peccati \& Tudor.
\begin{theorem}[Theorem 1 of \cite{peccati-tudor}]\label{PeTu}
Assume that for $q_1<q_2<\dots<q_m$ it holds that $\Esp\left[I_{q_i}(f^{(k)}_i)\right]^2\to_k\sigma^2_{ii}$ 
then the following conditions are equivalent:
\begin{enumerate}
 \item the vector  
\begin{equation*}
\left(I_{q_1}(f^{(k)}_1),\dots,I_{q_1}(f^{(k)}_m)\right)\,\mathop{\Rightarrow}\limits_{k\to\infty}\,N(0,D_m)
\end{equation*}
where $D_m$ is a diagonal matrix with entries $\sigma^2_{ii}$.
 \item for $i=1,\dots,m$ and $p=0,1,\dots,q_{i}-1$
\begin{equation*}
\|f^{(k)}_{i}\otimes_p f^{(k)}_{i}\|^2_{L^2([0,1]^{2q_i-2p})}\,\mathop{\to}\limits_{k\to\infty}\,0, 
\end{equation*}
where $\otimes_p$ denotes the $p$-order contraction, that is
\begin{multline*}
f^{(k)}_{q_i}\otimes_p f^{(k)}_{q_i}(x_1,\dots,x_{q_i-p},y_1,\dots,y_{q_i-p})\\=
\int_{[0,1]^p}f^{(k)}_{q_i}(x_1,\dots,x_{q_i-p},z_1,\dots,z_p)\\
\cdot f^{(k)}_{q_i}(y_1,\dots,y_{q_i-p},z_1,\dots,z_p)\;dz_1\;\dots\;dz_p .
\end{multline*}
\end{enumerate}
\end{theorem}

By Equation \eqref{f-q} we know that $I^{\overline{T}_K}_q([0,K\pi]_{-\alpha})=I^{B}_{q}(g_q)$ with 
\begin{equation*}
g_q({\boldsymbol\lambda},K)=\frac{1}{\sqrt{K\pi}}\int_{[0,K\pi]_{-\alpha}}\sum^{\left\lfloor q/2\right\rfloor}_{\ell=0}
b_{q-2\ell}a_{2\ell}(h(s,\cdot)^{\otimes q-2\ell}\otimes h^{\prime}(s,\cdot)^{\otimes 2\ell})({\boldsymbol\lambda})v_K(s)ds,
\end{equation*}
and ${\boldsymbol\lambda}=(\lambda_1,\dots,\lambda_q)$. 

Using the isometric property of the stochastic integral, 
we have
\begin{equation*}
h(s,\cdot)^{\otimes p}\otimes_n h(t,\cdot)^{\otimes p}
=\overline{r}_K(s,t)^{n}\cdot h(s,\cdot)^{\otimes p-n}\otimes h(t,\cdot)^{\otimes p-n}.
\end{equation*} 
Similar formulas hold for the terms which 
include $h^{\prime}$. 
More precisely, by the isometric property, they give factors that are powers of 
$\cov(\overline{T}_K(s),\mathcal {T}^{\prime}_K(t))=:\widetilde{r}^{\,\prime}_K(s,t)$ 
or $\cov(\overline{T}^\prime_K(s),\mathcal {T}^{\prime}_K(t))=:\widetilde{r}^{\,\prime\prime}_K(s,t)$. 
Therefore
\begin{multline*}
\|g_q(K)\|^2_2= 
\frac{1}{K\pi}\int_{[0,K\pi]_{-\alpha}}\int_{[0,K\pi]_{-\alpha}}v_K(s)v_K(t)\sum^{q/2}_{\ell=0}\sum^{q/2}_{\ell^\prime=0}
b_{q-2\ell}b_{q-2\ell^\prime}a_{2\ell}a_{2\ell^\prime}\\
\cdot\overline{r}_K(s,t)^{q-2(\ell\vee\ell^\prime)}{\widetilde{r}^{\,\prime\prime}_K(s,t)}^{2(\ell\wedge\ell^\prime)}
{\widetilde{r}^{\,\prime}_K(s,t)}^{2(\ell\vee\ell^\prime)-2(\ell\wedge\ell^\prime)}ds\; dt,
\end{multline*}
Where $x\vee y = \sup(x,y)$ and  $x\wedge y = \inf (x,y)$ .
It is quite tedious to write down the contractions and their norms, 
but the resulting integrals are quite similar. 
Let us do it for the case $n=1$ and $q$ odd (so that $q-2\ell>0$), we have
\begin{multline*}
\|g_q\otimes_1 g_q(K)\|^2_2= 
\frac{1}{(K\pi)^2}\int_{[0,K\pi]_{-\alpha}}\int_{[0,K\pi]_{-\alpha}}\int_{[0,K\pi]_{-\alpha}}\int_{[0,K\pi]_{-\alpha}}
\;ds\;dt\;ds^\prime\;dt^\prime\\
v_K(s)v_K(t)v_K(s^\prime)v_K(t^\prime)\sum^{q/2}_{\ell,\ell^\prime=0}\sum^{q/2}_{k,k^\prime=0}
b_{q-2\ell}b_{q-2\ell^\prime}a_{2\ell}a_{2\ell^\prime}b_{q-2k}b_{q-2k^\prime}a_{2k}a_{2k^\prime}\\
\cdot\overline{r}_K(s,t)\overline{r}_K(s^\prime,t^\prime)
\overline{r}_K(s,s^\prime)^{q-2(\ell\vee\ell^\prime)-1}{\widetilde{r}^{\,\prime\prime}_K(s,s^\prime)}^{2(\ell\wedge\ell^\prime)}
{\widetilde{r}^{\,\prime}_K(s,s^\prime)}^{2(\ell\vee\ell^\prime)-2(\ell\wedge\ell^\prime)}\\
\cdot\overline{r}_K(t,t^\prime)^{q-2(k\vee k^\prime)-1}{\widetilde{r}^{\,\prime\prime}_K(t,t^\prime)}^{2(k\wedge k^\prime)}
{\widetilde{r}^{\,\prime}_K(t,t^\prime)}^{2(k\vee k^\prime)-2(k\wedge k^\prime)}.
\end{multline*}
Observe that in the general case, 
the exponent of $\overline{r}_K(s,t)$ and $\overline{r}_K(s^{\prime},t^{\prime})$ is $n$, 
but the sum of the exponents in the factors involving $(s,t),(s,s^\prime)$, and 
in those involving $(s^\prime,t^\prime),(t,t^\prime)$, is $q$ 
(the total sum of the exponents is $2q$).

Now, we have to take the limit as $K\to\infty$. 

In the case of $\|g_q\|^2_2$, since $I^B_q$ is an isometry, 
it follows that $$\|g_q\|^2_2=\var\left(I^{\overline{T}_K}_{q}\left([0,K\pi]_{-\alpha}\right)\right)\to_{K\to\infty}\sigma^2_q(0)>0.$$

Now, we consider the limit of $\|g_q\otimes_ng_q\|^2_2$ as $K\to\infty$. 
Since, the $v_K$ are bounded by Equation \eqref{eq:vK2} and Inequalities \eqref{cota:supr}, 
the sums have a finite fixed number of terms and the $a$, $b$ coefficients are constant, 
the important ingredients are the covariances. 
We split the domain of integration into two parts: 
a tubular neighborhood of radio $\eta$ 
around the diagonal $s=t=s^\prime=t^\prime$ and its complement. 
Therefore, we have:

\noindent{\bf Close to the diagonal:} we assume that 
$|t-s|<\eta,|t^\prime-s^\prime|<\eta,|t^\prime-t|<\eta$ and $|s^\prime-s|<\eta$.

The covariances are bounded, in absolute value, from above by constants, 
for instance
\begin{equation*}
\left|\cov\left(\overline{T}_K(s),\mathcal{T}_{K}'(t)\right)\right|=\frac{1}{V_K(s)V_K(t)v_K(t)}
\left|\partial_t r_K(s,t)-\frac{c_K^\prime(2t)}{2V^2_K(t)}r_K(s,t)\right|,
\end{equation*}
with $r_K(s,t)=\textonehalf(c_K(\tau)+c_K(\sigma))$. 
At $\tau=0$ we have $c_K(0)=1$, $c^\prime_K(0)=0$ and $c^{\prime\prime}_K(0)=-\frac{(K+1)(2K+1)}{6K^2}$, 
thus, 
by continuity the terms involving $\tau$ are bounded. 
The remaining quantities are easily seen to be bounded by Lemma \ref{cotabaf} 
and Inequalities \eqref{cota:supr} and \eqref{cota:supvar}. 

Therefore, the integral is bounded by a constant times the volume of the tubular neighborhood 
of the diagonal. 
Such volume is proportional to $K\pi$. 

In conclusion, the $1/(K\pi)^2$ is not compensated by the integral, 
so the upper bound for the integral in the formula for the norm of the contraction 
(restricted to the neighborhood of the diagonal) tends to zero.

\noindent{\bf Far from the diagonal:} In the rest of the proof of the CLT, $C$ stands for some constant 
whose actual value is meaningless, but not depending on $K$.

We may assume that at least one of the following $|t-s|>\eta$, 
$|t^{\prime}-s^{\prime}|>\eta$, $|t^{\prime}-t|>\eta$, $|s^{\prime}-s|>\eta$ holds true. 

By Inequalities \eqref{cota:supr} and \eqref{cota:supvar}, the (absolute value of the) covariances 
$\overline{r}$, $\widetilde{r}^{\,\prime}$ and $\widetilde{r}^{\,\prime\prime}$ at $x,y$ are bounded from above 
by $C(1/|x-y|+1/|x+y|)$ when $|x-y|>\eta$. 
Therefore, the product in the integrand of $\|g_q\otimes_ng_q\|^2_2$ 
is bounded by the product of $1/|t-s|+1/|t+s|$, $1/|t^\prime-s^\prime|+1/|t^\prime+s^\prime|$, 
$1/|s-s^\prime|+1/|s+s^\prime|$ and $1/|t-t^\prime|+1/|t-t^\prime|$. 
Each one of these factors appears in the bound if the distance between 
the corresponding variables is larger than $\eta$. 
Otherwise, we bound them by constant as in the previous case.

Furthermore, the exponents of the covariances involving $s,t$ and $s^\prime,t^\prime$ 
is $n=1,\dots,q-1$, 
the sum of the exponents in the factors  involving $s,s^\prime$ is $q-n=1,\dots q-1\geq 1$, 
and the sum of the exponents in the factors  involving $t,t^\prime$ also is $q-n\geq 1$. 

Let us consider one of the possible cases, the others are similar. 
Say that $n=1$, $|t-s|>\eta$, $|t^{\prime}-s^{\prime}|<\eta$, $|t^{\prime}-t|<\eta$ and $|s^{\prime}-s|<\eta$, 
call $A$ the set of points verifying these inequalities. 
We bound the covariances involving the variables $t^\prime,s^\prime$; $t^\prime,t$; and $s^\prime,s$ 
in the integrand by constants. 
Then, the integrand is bounded by 
\begin{multline*}
C\int_A\left[\frac{1}{\tau}+\frac{1}{\sigma}\right]\;ds\;dt\;ds^\prime\;dt^\prime\\
=C\int_A\frac{1}{\tau}\;d\tau\;dt\;ds^\prime\;dt^\prime
+C\int_A\frac{1}{\sigma}\;d\sigma\;dt\;ds^\prime\;dt^\prime\\
\leq C\log(K\pi)\int_{A_{-1}}\;dt\;ds^\prime\;dt^\prime
\end{multline*}
where $\tau=t-s$, $\sigma=t+s$ and $A_{-1}=\{(t,s^\prime,t^\prime):|t^\prime-s^\prime|<\eta,|t^\prime-t|<\eta,|s^\prime-s|<\eta\}$.
Now, the volume of $A_{-1}$ is bounded by the volume of 
the tubular neighborhood of the diagonal of radius $\sqrt{3}\eta$, 
its volume is bounded by constant times $K\pi$, the result follows. 
Thus, the integral is bounded by a constant times $\log(K\pi)K\pi$, divided by $(K\pi)^2$, 
it tends to zero.

Putting together both parts we conclude that $\|g_q\otimes_ng_q\|^2_2\to 0$ as $K\to\infty$ 
for $n=1,\dots,q-1$. 
Therefore, $I^{\overline{T}_K}_q$ converges in distribution to a Gaussian random variable as $K\to\infty$
for all $q$.

\noindent{\bf Asymptotic Gaussianity of the sum:} 
Since, the $I^{\overline{T}_K}_{q}([0,K\pi]_{-\alpha})$ are orthogonal, by Theorem \ref{PeTu}
and Lemma \ref{conv:var}, 
their $Q$-th partial sum converges in distribution, as $K$ grows to infinity, to a Gaussian random variable 
with variance $\sum^{Q}_{1}\var\left(I^{X}_q\left([0,K\pi]_{-\alpha}\right)\right)$, 
therefore
\begin{equation*}
\frac{N_{\overline{T}_K}([0,K\pi]_{-\alpha})-\Esp\left(N_{\overline{T}_K}([0,K\pi]_{-\alpha})\right)}{\sqrt{K\pi}}=\sum^{\infty}_{q=1}I^{\overline{T}_K}_{q}([0,K\pi]_{-\alpha})
\end{equation*}
converges in distribution to a Gaussian random variable with variance 
$V^2=\sum^{\infty}_{q=1}\var\left(I^{X}_q\left([0,K\pi]_{-\alpha}\right)\right)$.
(It is well known that if $\mu_n\to\mu_\infty$, $\sigma_n\to\sigma_\infty$, $X_n\sim N(\mu_n,\sigma^2_n)$ 
then $X_n\Rightarrow X_\infty$.)
Finally, reasoning as in Lemma \ref{lemma:truncar}, note that 
$$\lim_{K\to\infty}\var\left(I^{X}_q\left([0,K\pi]_{-\alpha}\right)\right)=\lim_{K\to\infty}\var\left(I^{X}_q\left([0,K\pi]\right)\right)=\sigma^{2}_{q}(0).$$
This proves the Theorem.

\begin{proof}[Proof of Corollary \ref{cor}]
We do not fill in the details. 
The scheme of the proof is the following. 

First compute the cross correlation $\rho_K(s,t)=\Esp\left(\overline{T}_K(s)\overline{T}(t)\right)$
between $\overline{T}_K$ and $\overline{T}$:
\begin{multline*}
\rho_K(s,t):=\Esp{\overline{T}_K(s)\overline{T}(t)}\\
=\frac{1}{V_K(s)V(t)}
\Esp\int^1_0\sum^{K}_{n=1}\cos\left(\frac{n}{K}s\right)\indicator\left\{\left[\frac{n-1}{K},\frac{n}{K}\right)\right\}(\lambda)dB_{\lambda}
\cdot\int^1_0\cos\left(\lambda^{\prime} t\right)dB_{\lambda^{\prime}}\\
=\frac{1}{V_K(s)V(t)}\sum^{K}_{n=1}\int^{\frac{n}{K}}_{\frac{n-1}{K}}\cos\left(\frac{n}{K}s\right)\cos(\lambda t)d\lambda\\
=\frac{1}{V_K(s)V(t)}
\frac{1}{2}\sum^{K}_{n=1}\int^{\frac{n}{K}}_{\frac{n-1}{K}}\left[\cos\left(\frac{n}{K}s-\lambda t\right)+
\cos\left(\frac{n}{K}s+\lambda t\right)\right]d\lambda\\
=\frac{1}{\sqrt{1+c_K(2s)}\sqrt{1+\sinc(2t)}}\sum^{K}_{n=1}\int^{1/K}_{0}
\left[\cos\left(\frac{n}{K}\tau-vt\right)+\cos\left(\frac{n}{K}\sigma-vt\right)\right]dv.
\end{multline*}
where we made 
the change of variable $\lambda\mapsto\frac{n}{K}-\lambda$ in the last equality. 
It follows that, see Aza\"{i}s \& Le\'{o}n \cite{cltazaisleon}, $\rho_K(s,t)\longrightarrow \overline{r}(s,t)$ uniformly on 
off-diagonal compacts and so do the derivatives of $\rho_K(s,t)$ to the respective derivatives 
of $\overline{r}(s,t)$. 
Furthermore, these functions are bounded by $c(1/\tau+1/\sigma)$.

Now
\begin{equation*}
\Esp\left((I^{\overline{T}_K}_q-I^{\overline{T}}_q)^2\right)=
\Esp\left((I^{\overline{T}_K}_q)^2\right)-2\Esp\left(I^{\overline{T}_K}_q\cdot I^{\overline{T}}_q\right)
+\Esp\left((I^{\overline{T}}_q)^2\right).
\end{equation*}
Repeating the arguments in Proposition \ref{conv:var} and using these convergences we deduce 
that each expectation converges to $\sigma^2_q(0)$. 
Hence, $\|I^{\overline{T}_K}_q-I^{\overline{T}}_q\|_{L^2}\to_{K}0$.

Finally, repeating the arguments in the proof of Theorem \ref{main} for the Gaussianity of the sum 
we deduce that 
the asymptotic distributions of $N_{\overline{T}}([0,K\pi]_{-\alpha})$ 
and $N_{\overline{T}_K}([0,K\pi]_{-\alpha})$ (after standardization) coincide.
\end{proof}

\section{Proofs of the Lemmas}\label{proofs}

\begin{proof}[Proof of Lemma \ref{cotabaf}]
Let us start with the proof for $T$ and $\overline{T}^{\prime}$. 

The variance of $T$ is equal to $V(t):=(1+\sinc(2t))/2$, hence 
it suffices to prove that $\sinc(2t)>-2/\pi$ for $t\in[0,K\pi]$. 
Observe that the critical points of $\sinc$ are the roots of the equation $\sinc(x)=\cos(x)$, 
then, it is easy to see that in the interval $[-\pi/2,\pi,2]$ the only root of this equation is $x=0$, which 
is a maximum since $\sinc(0)=1$. 
Thus, the minimum of $\sinc$ lies outside this interval, 
so, as $|\sinc(x)|\leq 1/|x|$ the minimum is greater than $-2/\pi$, thus the result follows. 

Let us look now at the variance of $\overline{T}^{\prime}(t)$
\begin{equation}\label{var:Tprime}
v^2(t)=\frac{1}{1+\sinc(2t)}\left[\sinc^{\prime\prime}(2t)-\sinc^{\prime\prime}(0)-\frac{\sinc^{\prime}(2t)^2}{1+\sinc(2t)}\right].
\end{equation} 
Then
\begin{equation*}
v^2(t)\geq\frac{1}{2}\left[\frac{1}{3}-\sinc^{\prime\prime}(2t)-\frac{\sinc^{\prime}(2t)^2}{1+\sinc(2t)}\right]
\end{equation*}
and from the Inequalities in \eqref{cota:supr}, which hold since we are looking at compact intervals not containing $0$, 
it is easy to see that the latter expression is bounded away from zero for large $t$.

Now, let us consider the variances of $T_K$ and $\overline{T}^{\prime}_K$. 
By Lemma \ref{lema:cK} (see also Equation (37) in \cite{granville}) 
we know that as $K\to\infty$
\begin{equation*}
c^{(j)}_K(2s)=\sinc^{(j)}(2s)+O(1/K^{j}), 
\end{equation*}
where the constants involved in the $O$ notation do not depend on $K$ 
and $j=0,1,2$ stands for functional value, first and second derivative respectively. 
From this asymptotic and the result for $T$ and $\overline{T}^\prime$, the result 
follows also for $T_K$ and $\overline{T}^{\prime}_K$.
\end{proof}

\begin{proof}[Proof of Lemma \ref{lemma:truncar}]
Let us look at the interval $[0,(K\pi)^{\alpha}]$, the other one is analogous. 
First, we use Markov inequality to bound the probability by an expression involving the expectation of the number of roots, 
that is, for given $\varepsilon>0$ we have
\begin{equation*}
\Prob\left(\left|\frac{N_{\overline{T}_K}-\Esp\left(N_{\overline{T}_K}\right)}{\sqrt{K\pi}}\right|>\eps\right)
\leq\frac{\Esp\left|N_{\overline{T}_K}-\Esp\left(N_{\overline{T}_K}\right)\right|}{\eps\sqrt{K\pi}}
\leq\frac{2\Esp\left(N_{\overline{T}_{K}}\right)}{\eps\sqrt{K\pi}}
\end{equation*}
(note that above we have eliminated the dependence of the interval in the notation of crossings for ease of notation) thus, it is enough to show that $\Esp\left(N_{\overline{T}_K}\left([0,(K\pi)^{\alpha}]\right)\right)/\sqrt{K\pi}\to 0$.

With this aim, we use the first order Rice formula.
\begin{equation*}
\Esp\left(N_{\overline{T}_K}([0,(K\pi)^\alpha])\right)=
\int^{(K\pi)^{\alpha}}_0\Esp\left[|\overline{T}^{\prime}_K(t)|\mid\overline{T}_K(t)=0\right]p_{\overline{T}_K(t)}(0)dt. 
\end{equation*}
By a standard Gaussian regression 
\begin{equation*}
\Esp\left[|\overline{T}^{\prime}_K(t)|\mid\overline{T}_K(t)=0\right]
=\Esp\left[\left|\overline{T}^{\prime}_K(t)-\frac{\cov(\overline{T}^{\prime}_K(t),\overline{T}_K(t))}{\var(\overline{T}_K(t))}\overline{T}_K(t)\right|\right]. 
\end{equation*}
The random variable within the expectation is a centered Gaussian one, hence the expectation of its modulus equals 
$\sqrt{2/\pi}$ times its standard deviation, so
\begin{multline*}
\Esp\left(N_{\overline{T}_K}([0,(K\pi)^\alpha])\right)\\
=\frac{1}{\pi}\int^{(K\pi)^\alpha}_{0}\left[c_K^{\prime\prime}(2t)-c_K^{\prime\prime}(0)-
\frac{(c_K^\prime(2t)-c_K^\prime(0))^2}{1+c_K(2t)}\right]^{1/2}\frac{dt}{\sqrt{1+c_K(2t)}}
\end{multline*}
Let us look at the integrand. 
By Lemmas \ref{lema:cK} and \ref{cotabaf} we know that the denominator is bounded away from zero uniformly in $K$ 
and that the first factor is bounded uniformly in $K$ in any compact interval not containing $0$. 
Finally taking Taylor expansions at $t=0$, taking into account the particular values of the derivatives 
of $c_K$ at $0$ 
obtained from Equation \eqref{cK}, 
it follows that the first factor is bounded uniformly in $K$ 
in a neighborhood of $0$. 
%
Hence, the result follows. 
\end{proof}


\begin{proof}[Proof of Lemma \ref{cota:vardiag}]
We start from (\ref{cK}) that shows that $c_K$ has derivatives of any order uniformly bounded.

Let $t_0$ to be as in Lemma \ref{cotabaf}. We know that denominator of (\ref{def}) is bounded away from zero for $t,s\,\in[t_0,\pi-t_0],$ this implies that $\overline r_K(s,t)$ has bounded derivatives of any order. 

On $[t_0,\pi-t_0]$  let $d_K$ the distance induced by the sixth derivative $\overline T^{(6)}_K$ of $\overline T_K$ i.e.
\begin{eqnarray*}d^2_K(s,t)&:=&\Esp(\overline T^{(6)}_k(s)-\overline T^{(6)}_k(t))^2\\
&=&\overline r_K^{(6)}(s,s)+\overline r_K^{(6)}(t,t)-2\overline r_K^{(6)}(s,t)\\
&=&\frac{(s-t)^2}{2}\overline r_K^{(7)}(\zeta,\zeta)\le C(s-t)^2.
\end{eqnarray*}
This implies by the Dudley theorem, see for example  Theorem 2.40 of \cite{aw}, that $\displaystyle \Esp(\sup_{s\in t_0,\pi-t_0]}(\overline T^{(6)}_K(s))),$ is bounded and by consequence $\Esp||\overline T^{(6)}_K||_{\infty},$ where the norm is taken on $[t_0,\pi-t_0]$.

Then applying Theorem 3.6 of \cite{aw} with $m=2$ and $p=5$ we get that
$$\Esp((N_{\overline T_K}[t_0,t_0+a])^2)\le c_{5,2}[A+\mathbf{C}+\Esp||\overline T^{(6)}_K||_{\infty}],$$ where $\mathbf{C}$ is a bound for the density of $T_K(s)$ which is finite because the variance is bounded away from zero. 
Finally observe that these bounds are uniform in $K$ since the covariances and its derivatives are uniformly bounded in $K$. 
This concludes the proof.
\end{proof}

\begin{proof}[Proof of Lemma \ref{lemma:chaos}]
We drop the dependence of the intervals in the definition of the crossings.
\begin{enumerate}
\item This is a consequence of Lemma \ref{cota:vardiag}, which bounds the integral near the diagonal 
(the difficult part), or of Corollary 3.7 of Aza\"{i}s \& Wschebor \cite{aw}. 
\item To show this statement we begin with the following inequality that is a consequence of the Rolle's Theorem 
$$N_{\overline{T}_K}(u)\le N_{\overline{T}'_K}(0)+1,$$where $N_{\overline{T}_K}(u)$ are the crossings of $\overline{T}_K$ of level $u$ and $ N_{\overline{T}'_K}(0)$ are the zeros of the derivative of $\overline{T}_K$. But when $u$ is in a neighborhood of zero
we have that $\lim_{u\to0} N_{\overline{T}_K}(u)=N_{\overline{T}_K}(0)$ a.s (in fact the random variables are  locally constant). Thus by convergence dominated theorem and given that  $\Esp(N^2_{\overline{T}_K}(0))<\infty$, fact that can be proven similarly that the result of Lemma \ref{cota:vardiag}, we have
$$\lim_{u\to0}\Esp(( N_{\overline{T}_K}(u))^2)=\Esp(( N_{\overline{T}_K}(0))^2)<\infty.$$

\item First, note that $N^{\eta}_{\overline{T}_K}\to N_{\overline{T}_K}$ almost surely. 

Therefore, it suffices to state the convergence of the second moment. 

Fatou's Lemma implies that
$$\Esp\left((N_{\overline{T}_K})^2\right)=\Esp\left(\lim_{\eta}(N^{\eta}_{\overline{T}_K})^2\right)
\leq\lim_{\eta}\Esp\left((N^{\eta}_{\overline{T}_K})^2\right).$$
On the other hand, Area-Formula (see Federer \cite{federer} applied for $d=1$, 
$B=[0,t]$, $g=\varphi_{\eta}$ and $f=\overline{T}_K\in C^1$), almost surely, permits us to write 
\begin{equation*}
N^\eta_{\overline{T}_K}=\int^{\infty}_{-\infty}N_{\overline{T}_K}(u)\varphi_{\eta}(u)du=\Esp_Z\left(N_{\overline{T}_K}(Z)\right),
\end{equation*}
where $Z$ has density $\varphi_\eta$. 
Thus, applying Jensen's inequality in the inner expectation and Tonelli's Theorem, we have
\begin{multline*}
\Esp\left((N^{\eta}_{\overline{T}_K})^2\right)=\Esp\left((\Esp_Z(N_{\overline{T}_K}(Z)))^2\right)
\leq\Esp\left(\Esp_Z(N^2_{\overline{T}_K}(Z))\right)\\
=\int^{\infty}_{-\infty}\Esp\left( N^2_{\overline{T}_K}(u)\right)\varphi_{\eta}(u)du
\end{multline*}
Since, $\varphi_{\eta}$ approximates the unity, since $\Esp\left(N^2_{\overline{T}_K}(u)\right)$ is continuous in $u$ 
(part 2 of this lemma), 
passing to the limit, with $\eta\to 0$, in the latter inequality gives the desired result.\\

\item This fact follows from a straightforward adaptation of the proof of
Lemma 2 in Kratz \& Le\'{o}n \cite{kratz-leon-97}. 
In fact, let $\xi$ be a standard Gaussian random variable. 
Then, since $\varphi_\eta$ and $|\cdot|$ are functions in $L^2(\varphi(dx))$ they have Hermite expansions
\begin{equation*}
|\xi|=\sum^{\infty}_{k=0}a_kH_k(\xi);\quad
\varphi_\eta(\xi)=\sum^{\infty}_{k=0}b_kH_k(\xi)
\end{equation*}
where the Hermite coefficients are defined by $a_k=1/\sqrt{k!}\int|x|H_(x)\varphi(x)dx$ and 
$b_k=1/\sqrt{k!}\int\varphi_\eta(x)H_(x)\varphi(x)dx$. 
Next, since $\overline{T_K}(s)$ and $\mathcal T^{\prime}_K(s)$ are standard Gaussian random variables for each $s\in[0,K\pi]_{-\alpha}$ 
we can apply these expansions pointwise in $s$ and replace them in the integral in the r.h.s. of \eqref{Nsigma}. 

The rest of the proof (ie: passing to the limit under the integral sign) follows exactly the same lines as in \cite[Lemma 2]{kratz-leon-97}, 
we just have to add the bound for $v_K(s)$ which is easy.
\end{enumerate}
\end{proof}

\begin{proof}[Proof of lemma \ref{var:unifacot}]The proof follows the same lines of the proof of the similar claim in Aza\"{i}s \& Le\'{o}n \cite[Page 7]{cltazaisleon} with minor changes. 
The idea is to divide the integral into two parts, near the diagonal the bound is obtained using Cauchy - Schwarz Inequality 
and Lemma \ref{cota:vardiag}, 
for the off-diagonal part, the bound is obtained applying 
Arcones inequality \cite[Lemma 1]{arcones} for the vectors $(\overline{T}_{K}(s),\mathcal {T}^{\prime}_K(s))$ 
and $(\overline{T}_{K}(t),\mathcal {T}^{\prime}_K(t))$: if
\begin{multline*}
\psi_K(\tau,\sigma):=\max\{|\cov(\overline{T}_{K}(s),\overline{T}_{K}(t))|+
|\cov(\overline{T}_{K}(s),\mathcal {T}^{\prime}_K(t))|;\\
|\cov(\mathcal {T}^{\prime}_K(s),\overline{T}_{K}(t))|+
|\cov(\mathcal {T}^{\prime}_K(s),\mathcal {T}^{\prime}_K(t)|\}<1
\end{multline*}
then
\begin{equation*}
\Esp\left[f_q(\overline{T}_{K}(s),\mathcal {T}^{\prime}_K(s))
f_q(\overline{T}_{K}(t),\mathcal {T}^{\prime}_K(t))\right]\leq\|f_q\|^2_2 \psi^q_K.
\end{equation*}
Observe that Inequalities \eqref{cota:supr} allow to choose $\tau,\sigma$ large enough in order to have 
Arcones coefficient $\psi_K<\rho<1$, for fixed $\rho$. 
Besides $\|f_q\|^2_2$ is easily seen to be uniformly bounded using Hermite polynomials properties.\end{proof}
\section{Notation table}

\begin{tabular}{rr}
$ T_K(t)$ & see \eqref{f:tk}\\
$N_Z(I) $&  number of zeroes  of process $Z$ on $I$\\
$V$ & limit variance in Theorem  \ref{main} \\
$T(t)$  & see Corollary \ref{cor}\\
$r(s,t) $  &covariance of $T(.)$\\
$sc(x) $ & $\sin(x)/x$\\
$X_k(t)$ & see (\ref{Xk}) \\
$\widetilde T_K(t)$, $ \widetilde X_K(t),$ & see( \ref{f:tt})\\
$C_K(t)$ & see( \ref{cK}) \\
$ r_K(s,t)$ & see( \ref{rK}) \\
$ V^2_K(t)$ &$r_K(t,t)$\\
$\tau $ & $s-t$\\
$\sigma$ & $s+t$\\
$\overline{T}_K(t)$& see( \ref{def}) \\
$\overline{r}_K(s,t)$& see(\ref{barr})\\
$\overline{r}(s,t)$& see(\ref{limitscbar})\\
$B$ & Brownian motion\\
$\gamma_K(t,\lambda)$& see(\ref{f:gam})\\
$h_K(t,\cdot)$&$\gamma_K(t,\cdot)/V_K(t)$\\
$\mathcal T^{\prime}_K(s)$& see( \ref{f:cal}\\
$a_k$ , $b_k$ & see(\ref{f:akbk})\\
$f_q(x,y)$ &  see (\ref{f:fq})\\
$I^{\overline{T}_K}_{q}\left([0,K\pi]_{-{\alpha}}\right) $& see( \ref{f:iqtk})\\
$\varphi$,$\varphi_\eta$ &  density of the $N(0,1)$, $N(0,\eta)$\\
$N_{\overline{T}_K}(u,I)$  & number of crossings of  $u$ by $\overline{T}_K$ in the interval $I$\\
 $N^{\eta}_{\overline{T}_K,\alpha}$ & see (\ref{Nsigma})\\
 $C$ &some constant with meaningless value\\
 $v^2(t)$&  see (\ref{var:Tprime})
\end{tabular}

%
%
%


\end{document}